\documentclass[twoside,a4paper,reqno,11pt]{amsart} 
\usepackage[top=30mm,right=30mm,bottom=30mm,left=30mm]{geometry}
\usepackage[utf8]{inputenc}
\usepackage{amsthm}
\usepackage{amsfonts}
\usepackage{amsmath}
\usepackage{ amssymb }
\usepackage{url}
\usepackage{todonotes}
\usepackage{hyperref} 
\usepackage{enumerate}%
\usepackage{comment}
\usepackage[foot]{amsaddr}
\newtheorem{theorem}{Theorem}
\newtheorem{proposition}[theorem]{Proposition}
\newtheorem{corollary}[theorem]{Corollary}
\newtheorem*{theorem*}{Theorem}
\newtheorem{question}[theorem]{Question}
\newtheorem*{conjecture*}{Conjecture}

\newtheorem{lemma}[theorem]{Lemma}
\theoremstyle{definition}

\newtheorem*{definition*}{Definition}
\newtheorem*{proposition*}{Proposition}
\newtheorem{example}[theorem]{Example}

\newtheorem{remark}[theorem]{Remark}
\numberwithin{equation}{section}
\usepackage[numbers,sort]{natbib}

\title{on the regular power structure of $p$-groups and applications}
\author{James Williams }
\address{J. Williams, School of Mathematics, University of Bristol, Bristol BS8 1UG, UK}
\email{j.l.i.williams@bristol.ac.uk}
\date{\today}

\begin{document}

\maketitle

\begin{abstract}
    In this paper, we give elementary proofs of the Restricted Burnside Problem and the Hughes Conjecture for finite $p$-groups with Hall's regular power structure property. Moreover, in this setting we determine an explicit bound on the order of a finite $d$-generator $p$-group of fixed exponent. Further applications of $p$-groups with regular power structure are presented. For example, we give a short new proof of an important property of powerful $p$-groups; namely, that the minimal number of generators of a subgroup of such a group $G$ is at most the number needed to generate $G$.
\end{abstract}

\section{Introduction}
It is widely recognised that the world of finite $p$-groups is a complicated one. However, perhaps surprisingly, most finite $p$-groups have many properties in common with abelian groups. In his landmark paper of 1933 \cite{phallcontribtothetheoryofgroupsofprimepowerorder}, Philip Hall began to draw out the analogy between a certain class of finite $p$-groups and abelian $p$-groups. He observed that \textit{regular} $p$-groups have three specific properties, which are all  satisfied by abelian groups. Groups satisfying these three properties are said to have a \textit{regular power structure}. 
\begin{definition*} 
A finite $p$-group $G$ has a \textit{regular power structure} if the following three conditions hold for all positive integers $i$:
\begin{align}
 G^{p^{i}} &=\{g^{p^{i}} \mid g \in G \} \label{eqn pth power} \\
      \Omega_{i}(G) &= \{g\in G \mid o(g) \leq p^{i} \} \label{eqn omega conditon} \\
 |G:G^{p^{i}}| &= | \Omega_{i}(G)| \label{eqn index condition}
\end{align}
where $\Omega_{i}(G) = \langle g \in G \mid o(g) \leq p^{i} \rangle.$
\end{definition*}

In the decades that followed many authors have studied the power structure of finite $p$-groups \cite{mannpowerstructure1, mannpowerstructureofpgroupsii,mingyaosurvery, kluempenpowerstructureof2groups} and numerous families of groups with a regular power structure  have been identified. For instance, as well as regular $p$-groups, it is known that powerful $p$-groups \cite{wilsonsphd, L2002, Fernandez-Alcober2007}, potent $p$-groups \cite{GONZALEZSANCHEZ2004193} and quasi-powerful $p$-groups \cite{williams2019quasipowerful}, all have a regular power structure for every odd prime $p$. 

Thus we see there is a very large family of groups  with these desirable, abelian-like properties. Since the order $p^{n}$ of a finite $p$-group depends on both the prime $p$ and the integer $n$, there are two ways in which we can demonstrate the vastness of the family of $p$-groups with regular power structure. On the one hand, if $n$ is fixed then for any prime $p$ which is greater than $n$, any group of order $p^{n}$ will be regular, and hence have a regular power structure  \cite[Corollary 12.3.1]{hall1999theory}. On the other hand if $p>2$ is a fixed prime then by the Higman-Sims bound (\cite{higmanenumeratingpgroups1960}, \cite{simsenumeratingpgroups})  the number of groups of order $p^{n}$ is $p^{2/27n^{3}+O(n^{-1/3})}$. In \cite{higmanenumeratingpgroups1960} a lower bound on the number of $p$-groups is obtained by exhibiting a family of groups of order $p^{n}$. The size of this family is $p^{2/27 n^{3} + O(n^{-1/3})}$. The groups in this family have nilpotency class $2$, and thus as $p$ is odd they are regular and so have a regular power structure. In particular, as $n \rightarrow \infty$ the probability that a group of order $p^{n}$ has a regular power structure tends to $1$.

In this paper our first goal is to demonstrate the utility of identifying a large family of groups with such desirable properties. To do this, we will show that two famous problems can be answered relatively easily when the problems are restricted to groups with regular power structure. Namely the Restricted Burnside Problem and the Hughes Conjecture.

Recall that the Restricted Burnside Problem asks whether or not there are only finitely many finite $d$-generator groups of exponent $e$. A result by Hall and Higman in 1956 effectively reduced the problem to the case where the exponent is a prime power  \cite{hallhigman1956}. In \cite{Kostrikin} Kostrikin showed that the problem has an affirmative answer in the case that the exponent is a prime. In 1989 Zelmanov announced a positive solution to the Restricted Burnside Problem \cite{zelmanov,zelmanov2}. For an excellent history of the Burnside problems, we recommend \cite{historyburnside}.

We give a short, novel proof that establishes an affirmative answer to the Restricted Burnside Problem for finite $p$-groups with a regular power structure, assuming only Kostrikin's result for prime exponent. Our method is elementary and does not rely on Lie Ring methods.

\begin{theorem}
\label{Theorem intro: Burnside}
Let $G$ be a finite $p$-group with regular power structure, exponent $p^{e}$ and $d$ generators. Then there is a bound on the order of $G$, depending only on $p,e$ and $d$.
\end{theorem}
In other words, there is a largest finite, $d$-generator $p$-group with exponent $p^{e}$ and regular power structure. 
Furthermore, it's worth noting that we are able to explicitly bound the order of the group in terms of the order of the largest $d$-generator group of exponent $p$. We give an infinite family of groups which attains the bound (see Example \ref{example family of groups}). 

We next turn our attention to the Hughes conjecture (originally posed in \cite{hughesconjectureposed}). 

\begin{conjecture*}[Hughes \cite{hughesconjectureposed}]
Let $G$ be a group, $p$ a prime and let $H_{p}(G)$ be the subgroup generated by the elements in $G$ that do not have order $p$. If $H_{p}(G) \neq 1$, then either $H_{p}(G)=G$ or $|G:H_{p}(G)|=p$.
\end{conjecture*}

 There are many settings where the conjecture has been established (see Section \ref{section Hughes}), but it is known to be false in general. We extend the well known result that regular $p$-groups satisfy the Hughes conjecture to groups with regular power structure. Moreover we can determine $H_{p}$ precisely. 

\begin{theorem} \label{Theorem Intro: Hughes Conjecture}
If $G$ is a finite $p$-group with a regular power structure, then $G$ satisfies the Hughes conjecture. In particular, if $G$ has exponent $p$ then $H_{p}(G)=1$, otherwise $H_{p}(G)=G$.
\end{theorem}

We next take a different viewpoint. We look at a setting where, perhaps surprisingly, groups with a regular power structure occur naturally. 
Taking advantage of the fact that for odd primes metacyclic $p$-groups are powerful, and thus enjoy many nice properties including a regular power structure, we give a new proof of the following classical result. 

\begin{theorem}
\label{theorem into: G cyclic if all normal abelian subgroups are cyclic}

Let $p$ be an odd prime and $G$ a finite $p$-group with the property that every normal abelian subgroup is cyclic. Then $G$ is cyclic. 
\end{theorem}

This is a rather innocuous looking statement, but the typical proof of this result can be quite involved \cite[Theorem 4.10]{gorensteintheory}. Our proof makes use of the fact that powerful $p$-groups arise naturally in this setting.

Powerful $p$-groups are in many ways very similar to abelian groups. Indeed, not only do they have a regular power structure when $p$ is odd, but in fact for any prime they have the property that the minimal number of generators of a subgroup cannot exceed the minimal number of generators of the group itself \cite[Theorem 1.12]{LUBOTZKY1987484}. This is one of the most important properties of powerful $p$-groups.

To conclude this note we present a new, short and elementary proof of this important fact, which relies only on basic properties of powerful $p$-groups. \vspace{3mm}

\noindent \textbf{Notation:} Our notation is standard. We denote the order of $x\in G$ as $o(x)$.  All iterated commutators are left normed. The terms of the \textit{lower central series} of $G$ are defined recursively as $\gamma_{1}(G)=G$ and $\gamma_{k+1}(G)=[\gamma_{k}(G),G]$ for integers $k\geq 1$. The $i$th term of the \textit{upper central series} of a group $G$ is denoted  $Z_{i}(G)$. We use bar notation for images in a quotient group; it will always be made explicitly clear what the quotient group under consideration is. We denote the minimal number of generators of a finite group $G$ by $d(G)$. The cyclic group of order $n$ is denoted by $C_{n}$. The Frattini subgroup of $G$ is denoted by $\Phi(G)$.

\section{Restricted Burnside Problem}
\setcounter{theorem}{0}
\numberwithin{theorem}{section}
In this section we show that there are only finitely many finite $m$ generator groups of exponent $p^{e}$ with regular power structure. 
We will make use of the result of Kostrikin \cite{Kostrikin}, which states that there exists a largest finite $d$-generator group of exponent $p$. We shall denote the order of the largest finite $d$-generator group of exponent $p$ as $n_{(d,p)}$ 

\begin{theorem}
\label{theorem burnside more detailed with bound}
Given a finite $p$-group $G$ with $d$ generators and exponent $p^{e}$ there is a bound on the order of $G$ depending only on $p,e$ and $d$. 
In particular $|G|\leq n_{(d,p)}{}^{e}.$
\end{theorem}
\begin{proof}
We proceed by induction on the exponent. The base case when the exponent is $p$ is dealt with by Kostrikin's theorem.

Now suppose that $G$ is a $d$-generator group with regular power structure, exponent $p^{k}$ and that the claim holds for $d$-generator groups with regular power structure of smaller exponent.

Notice that $G/G^{p^{e-1}}$ is a $d$-generator group with regular power structure, of exponent $p^{e-1}$. Therefore by the inductive hypothesis we have 
\begin{equation}
\label{eqn bound on G/Gpe-1}
    |G/G^{p^{e-1}}|  \leq n_{(d,p)}{}^{e-1}.
\end{equation}
Next we claim that $G^{p^{e-1}}$ is contained in $\Omega_{1}(G)$. To see this, notice that  by the first regular power structure condition (\ref{eqn pth power}) each element of $G^{p^{e-1}}$ can be written in the form $g^{p^{e-1}}$ for some $g \in G$ and hence is of order $p$. Hence 
\begin{equation}
\label{eqn Gpe-1 in omega 1}
   G^{p^{e-1}} \leq \Omega_{1}(G).  
\end{equation}

Finally notice that by the third condition (\ref{eqn index condition}) we have 
\begin{equation}
\label{eqn bound on omega1}
    |\Omega_{1}(G)| = |G/G^{p}| \leq n_{(d,p)},
\end{equation} 
since $G/G^{p}$ is a $d$-generator group with exponent $p$. 
Putting these three equations together we have that 
\begin{align*}
    |G| &=|G/G^{p^{e-1}}| | G^{p^{e-1}}| && \\
        &\leq |G/G^{p^{e-1}}| |\Omega_{1}(G)| &&\text{by (\ref{eqn Gpe-1 in omega 1}}) \\
        &\leq n_{(d,p)}{}^{e-1}\cdot n_{(d,p)} &&\text{by (\ref{eqn bound on G/Gpe-1}) and (\ref{eqn bound on omega1})}\\
        &=n_{(d,p)}{}^{e}. &&
\end{align*}
This completes the proof. 
\end{proof}
Thus we have established Theorem \ref{Theorem intro: Burnside} as stated in the introduction. We present an infinite family of $2$-generator $3$-groups in order to demonstrate that the bound on the order of $G$ in Theorem \ref{theorem burnside more detailed with bound} is sharp. 
\begin{example}
\label{example family of groups}
 Consider the following group given by the presentation $$ G_{e}=\langle a,b,c \mid a^{3^{e}}, b^{3^{e}}, c^{3^{e}}, [c,a], [c,b], [b,a]=c \rangle. $$
It is easy to see that $G_{e}$ is a semidirect product of the form $ (C_{3^{e}} \times C_{3^{e}}) \rtimes C_{3^{e}}$ and so is of order $3^{3e}$. Moreover notice that $G_{e}$ has nilpotency class $2$ and therefore is regular. It follows that the exponent is $3^{e}$.

It is well known that the largest $2$-generator group of exponent $3$ is of order $27$ therefore $n_{(2,3)}=27$  \cite[Theorem 18.2.1]{hall1999theory}. Hence Theorem \ref{theorem burnside more detailed with bound} implies that the largest $2$-generator group of exponent $3^{e}$ has order at most $n_{(2,3)}^{e}=27^{e}=3^{3e} $. The group $G_{e}$ demonstrates that the bound is attained.
\end{example}

We conclude this section by presenting a question which suggests an additional motivation for Theorem \ref{theorem burnside more detailed with bound}.

\begin{question}
\label{question about burnside}
Does there exist a function $f(p,e,d)$ such that every finite $d$-generator $p$-group of exponent $p^{e}$ has a subnormal series of length at most $f(p,e,d)$ and in which all of the factors have regular power structure?
\end{question}

By \cite[Lemma 4.2.2]{hallhigman1956} in the celebrated paper of Hall and Higman, an affirmative answer to Question \ref{question about burnside} implies an affirmative  solution to the Restricted Burnside Problem for groups of prime power exponent. 

\section{Hughes Conjecture} 
\label{section Hughes}
In this section we study the Hughes conjecture,  in the context of groups with a regular power structure. We begin by recalling the Hughes conjecture,  originally posed in \cite{hughesconjectureposed}.

\begin{conjecture*}[Hughes \cite{hughesconjectureposed}]
Let $G$ be a group, $p$ a prime and let $H_{p}(G)$ be the subgroup generated by the elements in $G$ that do not have order $p$. If $H_{p}(G) \neq 1$, then either $H_{p}(G)=G$ or $|G:H_{p}(G)|=p$.
\end{conjecture*}

We list a few instances where the conjecture holds:

\begin{enumerate}[(i)]
    \item Any group $G$ and prime $p=2$ or $p=3$  \cite[Lemma 4]{hughespartialdifferencesets},\cite{strausszekeresonhughes}.
    
    \item $G$ a finite group which is not a $p$-group \cite{hughesthompson}.
    
    \item $G$ a finite metabelian $p$-group \cite{HoganKappe}.
    
    \item $G$ a finite $p$-group with nilpotency class at most $2p-2$ \cite{macdonaldhughesproblemclass2p-2}.
\end{enumerate}

There are many other settings where the conjecture has been established (for example see \cite{gallian, gallian2, gallian3}). 

However the conjecture is false in general. The first counterexample, which was constructed by Wall in \cite{wall}, was a $3$-generator finite $5$-group $G$ with $|G:H_{p}(G)|=25$. Further counterexamples have since been constructed and the conjecture is now known to be false for primes $ 5 \leq p \leq 19$ (\cite{oncounterexamplestothehughesconjecture,khukhrohughes1, khukhrohughes2}). It is expected that counterexamples exist for all primes $p>3$.

On the other hand, in \cite{khukhrohughes3} Khukhro proved the remarkable result that the Hughes conjecture is true for almost all finite $p$-groups. In particular, for a given prime $p$, if the $d$-generator group $G$ is a counterexample to the Hughes conjecture, then $|G| \leq p^{\beta(d,p)}$ where $\beta$ is some function depending only on $d$ and $p$.

In this section we shall establish Theorem \ref{Theorem Intro: Hughes Conjecture}, that the Hughes conjecture is satisfied by the family of finite $p$-groups with regular power structure.

It is known that if $G$ is a regular $p$-group then the Hughes conjecture holds. Indeed, if the exponent of $G$ is $p$ then $H_{p}(G)=1$, and if the exponent of $G$ is greater than $p$ then there is an element $a \in G$ of order $p^{2}$. Let $b$ be any element of $G$ of order $p$. As the group is regular we have that 
$$(ab)^p = a^p b^p c^p = b^p c^p$$
for some $c \in \langle a ,b \rangle ^{\prime}$. However 
$$ o([a,b])=o((b^{-1})^{a}\cdot b) \leq p $$ by property (\ref{eqn omega conditon}) and the fact that regular $p$-groups have a regular power structure.
Thus
$$(ab)^p=b^p.$$
In particular $ab$ does not have order $p$ and so both $ab$ and $b$ are in $H_{p}(G)$, and consequently $a \in H_{p}(G)$. We note that for a regular $p$-group, the case that $H_{p}(G)$ has index $p$ does not occur (except trivially for $G=C_{p}$). 

We will show that these observations can be extended to any group with a regular power structure. We thus turn to the proof of Theorem \ref{Theorem Intro: Hughes Conjecture}.

\begin{proof}[Proof of Theorem \ref{Theorem Intro: Hughes Conjecture}]
If $G$ has exponent $p$ then $H_{p}(G)=1$. Thus we can now consider the case when $G$ has exponent strictly greater than $p$. In this case we know that any generating set for $G$ must contain an element of order greater than $p$, or else by property (\ref{eqn omega conditon}) the exponent of $G$ would be $p$. Let $G= \langle a_{1}, \dots, a_{s}, b_{1}, \dots, b_{t} \rangle$ where the $a_{i}$'s have order $p$ and the $b_{j}$'s have order greater than $p$. 

We will show that each $a_{i} \in H_{p}(G)$. Consider for some $i \in \{1,\dots, s \}$ the products $a_{i}b_{1}, \dots, a_{i}b_{t}$ and notice that $\langle a_{1}, \dots, a_{s}, a_{i}b_{1}, \dots, a_{i}b_{t} \rangle$ is a generating set for $G$. If each of the products $a_{i}b_{1}, \dots, a_{i}b_{t}$ were of order $p$, then by regular power structure property (\ref{eqn omega conditon}) it follows $G$ has exponent $p$, a contradiction. Thus we must have that for some $j$, the product $a_{i}b_{j}$ has order greater than $p$. Then it follows that $a_{i}b_{j}$ and $b_{j}$ are both in $H_{p}(G)$ and thus $a_{i} \in H_{p}(G)$. Hence $H_{p}(G)=G$.
\end{proof}

In particular we have the following.

\begin{corollary}
Let $G$ be a finite $p$-group with regular power structure. If $G$ has exponent $p$ then $H_{p}(G)=1$, otherwise $H_{p}(G)=G$.
\end{corollary}

Notice that for groups $G$ with a regular power structure, we have that $[G:H_{p}]=p$ if and only if $G=C_{p}$.

\section{On \texorpdfstring{$p$}{p}-groups with every normal abelian subgroup cyclic}

It is a classical result that if $p$ is an odd prime and $G$ is a $p$-group such that every normal abelian subgroup is cyclic, then $G$ itself is cyclic. In this section we provide an alternative proof of this fact, making use of the appearance of metacyclic $p$-subgroups and their regular power structure for odd primes. This relies on the observation that for odd primes $p$, metacyclic $p$-groups are powerful, and so have regular power structure. 

Powerful $p$-groups, introduced in \cite{LUBOTZKY1987484}, appear throughout the rest of this paper.  Recall that a finite $p$-group $G$ is said to be \textit{powerful} if $[G,G]\leq G^{p}$ in the case that $p$ is odd, and $[G,G] \leq G^{4}$ in the case that $p=2$.

In this paper we make use of the following basic facts about powerful $p$-groups, often without explicit mention. 
\begin{proposition}
Let $G$ be a powerful $p$-group and $i,j \geq 0$, then
\begin{enumerate}[(i)]
    \item $G^{p^{i}}=\langle g^{p^{i}} \mid g \in G \rangle = \{ g^{p^{i}} \mid g \in G \}.$
    \item $[G^{p^{i}},G^{p^{j}}]\leq[G,G]^{p^{i+j}}.$
\end{enumerate}
\end{proposition}

For proofs of these facts and a textbook exposition on powerful $p$-groups we strongly recommend Chapter 11 of \cite{khukhro_1998}.

For completeness we include a proof of the fact that for any odd prime $p$, metacyclic $p$-groups are powerful.

\begin{lemma}
\label{lemma metacyclic pgroups are powerful for odd p}
Let $p$ be an odd prime and $G$ a finite $p$-group. If $G$ is metacyclic then $G$ it powerful. 
\end{lemma}
\begin{proof}
 First observe that in any group, an element of order $p$ cannot be conjugate to a (proper) power of itself. This follows from the N/C Theorem: If $H \leq G$ then $\frac{N_{G}(H)}{C_{G}(H)}\cong K \leq \text{Aut}(G),$ and so anything that normalises a group of order $p$ must in fact centralise it. 
 
 Now consider a metacyclic $p$-group $G=\langle a, b \rangle$ where $\langle b \rangle $ is normal in $G$. Then we claim that $\bar{G}=G/G^{p}$ is abelian, since $\bar{b}^{\bar{a}} \leq \langle \bar{b} \rangle$, and by the observation above we must have that $\bar{b}^{\bar{a}}=\bar{b}.$ Hence $[G,G] \leq G^{p}$ and if $p$ is odd then $G$ is powerful.
\end{proof}
\begin{remark}
\begin{enumerate}[(i)]

    \item In fact, in any $p$-group $G$, if $N$ is a cyclic normal subgroup then by considering $G/N^{p}$ the same argument as in the proof of Lemma \ref{lemma metacyclic pgroups are powerful for odd p} shows that $[N,G] \leq N^{p}$. 

    \item The fact that an element of order $p$ cannot be conjugate to a power of itself is a key ingredient in \cite{HobbyFrattini}, where Hobby proves that non-abelian groups with cyclic center can never occur as the Frattini subgroup of a finite $p$-group.
\end{enumerate}
\end{remark}

We now turn our attention to proving Theorem \ref{theorem into: G cyclic if all normal abelian subgroups are cyclic}. We shall need the following lemma. 
\begin{lemma}
\label{lemma adjusting generators for 2 gen pow p group}
Let $G=\langle a,b \rangle$ be a powerful $p$-group such that $G^{p}=\langle a^{p} \rangle$. Then there exists an element $c \in G$ of order $p$ such that $G = \langle a, c \rangle$. 
\end{lemma}
\begin{proof}
We prove the claim by induction on the exponent of $G$. The result is clearly true when the exponent is $p$. Now suppose the exponent is $p^{k+1}$ and that the claim holds for smaller exponent. Consider the quotient $\bar{G}=G/G^{p^{k}} = \langle \bar{a}, \bar{b} \rangle$. Then $\bar{G}$ satisfies the inductive hypothesis and is of smaller exponent, hence there exists an element $\bar{d} \in \bar{G}$ with $o(\bar{d})=p$ and $\bar{G}=\langle \bar{a}, \bar{d}\rangle$. Then $G= \langle a, d \rangle$ and $d^{p} \in G^{p^{k}}= \langle a^{p^{k}} \rangle $. Then we have that $d^{p} = a^{\lambda p^{k}}$ for some $0 \leq  \lambda <p$. Let $c=d a^{-\lambda p^{k-1}}$. By observing that $[G,G^{p^{k-1}}]\leq G^{p^{k}} \leq Z(G)$ and that $[G,G^{p^{k-1}}]^{p}=1$ we see that $c^{p}=1$ as required.  
\end{proof}

We are now in a position to prove Theorem \ref{theorem into: G cyclic if all normal abelian subgroups are cyclic}.
\begin{proof}[Proof of Theorem \ref{theorem into: G cyclic if all normal abelian subgroups are cyclic}]
Let $p$ be an odd prime and $G$ be a finite $p$-group with the property that every normal abelian subgroup of $G$ is cyclic. We shall show that $G$ is cyclic. The result is clear in the case that $|G|=p$. Thus we can suppose $|G|\geq p^{2}$. Notice that if $a \in Z_{2}(G)$ then $H=\langle a, Z(G) \rangle$ is an abelian normal subgroup and thus cyclic. Hence we can assume that $G$ contains a cyclic normal subgroup of order at least $p^{2}$. We will show that the existence of a cyclic normal subgroup of order $p^{k} \geq p^{2}$ implies the existence of one of order $p^{k+1}$. It  then follows that $G$ is cyclic.

Suppose that $H= \langle a \rangle $ is a cyclic normal subgroup of order $p^{k} \geq p^{2}$. If $G=H$ we are done, otherwise as $G$ is a $p$-group there exists a normal subgroup $N$ of $G$ with order $p^{k+1}$ such that $N \geq H $. We will show that $N$ is cyclic. Notice that $N$ is a metacyclic $p$-group (since $N/H$ is cyclic), and thus it is a powerful $p$-group and we shall write $N= \langle a, b \rangle$. If $N$ is cyclic we are done, if not then we have $N^{p}=\langle a^{p} \rangle$ and so by  Lemma  \ref{lemma adjusting generators for 2 gen pow p group}  we can find $c\in N$ of order $p$ such that $N= \langle a, c \rangle $. As $[c,a]=c^{-1}c^{a}$ is the product of two elements of order $p$ and the group has a regular power structure since it is powerful, then $o([c,a]) \leq p$. Thus $N^{\prime}$ is cyclic of order at most $p$ and $N$ has nilpotency class at most $2$, and so we  have that $[c,a]^{p}=[c^p,a]=1$ and so $J= \langle a^{p}, c \rangle $ is abelian. Our next step is to show that $J$ is normal in $G$.

As $H=\langle a \rangle$ is normal in $G$, it follows that $\langle a^{p} \rangle$ is normal in $G$. Next notice that for any $g \in G$, since $N$ is normal in $G$ we have that $c^{g} \in N$. We have that $c^{g}$ is of order $p$ and thus $c^{g} \leq \Omega_{1}(N) \leq \langle a^{p},c \rangle $, since by the second regular power structure property (\ref{eqn omega conditon}), the exponent of $\Omega_{1}(N) \leq p$. Thus it follows that $J$ is normal in $G$. Hence $J$ must be cyclic and by order considerations of $a^{p}$ and $c$ we must have $c \in \langle a^{p} \rangle $ and thus $N= \langle a , c \rangle = \langle a \rangle $ is cyclic.
\end{proof}

\section{Minimal generation and powerful groups}
In this final section we give a short and elementary proof of a key result in the theory of $p$-groups, that for a powerful $p$-group $G$ we have that $d(H) \leq d(G)$ for all $H \leq G$. Our argument makes use of only one of the regular power structure properties, property (\ref{eqn pth power}).

We will need the following lemma, which is easily proved by basic properties of powerful $p$-groups and the commutator collection formula of P. Hall. 
\begin{lemma}
\label{lemma powerful p groups are p^e-1 abelian}
Let $G$ be a powerful $p$-group of exponent $p^{e}$ and $g,h \in G$. Then $(gh)^{p^{e-1}}=g^{p^{e-1}}h^{p^{e-1}}.$
\end{lemma}
\begin{proof}
We use the following formulation of P. Hall's collection formula. If $G$ is a group, $x,y \in G,$ and $n \in \mathbb{N}$ then  
\begin{equation}
\label{p hall collection  (xy)^p^n}
(xy)^{p^{n}} \equiv x^{p^{n}}y^{p^{n}} \left(\thinspace \text{mod} \thinspace \gamma_{2}(T)^{p^{n}}\gamma_{p}(T)^{p^{n-1}}\dots \gamma_{p^{n}}(T)\right)
\end{equation} where $T= \langle x,y \rangle$.
Then for a powerful $p$-group $G$ we have that $\gamma_{i}(G)\leq G^{p^{i-1}}$ (this is true for any prime $p$ including $p=2$). It is then clear that if $G$ has exponent $p^{e}$ and $n=e-1$ that each term in the congruence of (\ref{p hall collection  (xy)^p^n}) is trivial. Hence $(xy)^{p^{e-1}}=x^{p^{e-1}}y^{p^{e-1}}$. 
\end{proof}
\begin{remark}
We remark that in \cite{HobbyCharSubgroupOfPGroup}, Hobby introduced the notion of a $p$-group being \textit{$p$-abelian} if for any $a,b \in G$ we have that $(ab)^{p}=a^{p}b^{p}$, thus in this spirit we could say that powerful $p$-groups are $p^{e-1}$-abelian.
\end{remark}

Recall that for any finite $d$-generator $p$-group $G$, $G/\Phi(G)$ can be thought of as a $d$ dimensional vector space over $\mathbb{F}_{p}$. We say that elements $a,b \in G$ are \textit{linearly
independent} over $\Phi(G)$ if their images are linearly independent in $G/\Phi(G)$. Notice that for a powerful $p$-group $\Phi(G)=G^{p}$. 

We can now prove the following.

\begin{theorem}
\label{Theorem generating set size for subgroup lifts up}
Let $G$ be a powerful $p$-group and $H \leq G$ with $d(H)=r$. Then $G$ contains a set of $r$ linearly independent elements over $G^{p}$.
\end{theorem}
\begin{proof}
 Suppose $d(H)=r$ and express the generating set of $H$ maximally as $p$th powers, so $H=\langle a_{1}^{p^{b_{1}}}, \dots, a_{r}^{p^{b_{r}}} \rangle$ and $a_{i} \notin G^{p}$ for each $i$. Then we claim that $a_{1}, \dots, a_{r}$ are linearly independent over $G^{p}$.
 It is clear that the claim is true for powerful $p$-groups of exponent $p$ as these groups are abelian. 
 
 Let $G$ be a powerful $p$-group with exponent $p^{e} \geq p^{2}$. Suppose that the claim holds for all groups of smaller order.  There is some central element in $G^{p^{e-1}}$ of order $p$, call this element $z$;  notice that this element can be written as $g^{p^{e-1}}$ for some $g \in G$. Consider the quotient $\bar{G} = G/\langle z \rangle,$ and let $\bar{H}$ be the image of $H$ under the natural map. 
 
 If $d(\bar{H})=r$, then $\bar{a_{1}}, \dots, \bar{a_{r}}$ are linearly independent over $\bar{G}^{p}= \frac{G^{p}}{\langle z \rangle}$, thus $a_{1}, \dots, a_{r}$ are linearly independent over $G^{p}$.
 
 Now assume $d(\bar{H}) \neq r$. Then as we have quotiented out by a group of order $p$, we must have $d(\bar{H})=r-1$. Then we may assume that some generator of $H$ is $z=g^{p^{e-1}}$. So $H = \langle a_{1}^{p^{b_{1}}}, \dots, a_{r-1}^{p^{b_{r-1}}}, g^{p^{e-1}} \rangle$. As above, we use the induction hypothesis to deduce that $a_{1}, \dots, a_{r-1}$ are linearly independent over $G^{p}$. Finally we claim that in fact $a_{1}, \dots, a_{r-1}, g$ are linearly independent over $G^{p}$. 
 Suppose for contradiction that  $g = a_{1}^{\lambda_{1}} \dots a_{r-1}^{\lambda_{r-1}} x^{p} $ for some non negative integers $\lambda_{i}$ and some $x \in G$. Then by Lemma \ref{lemma powerful p groups are p^e-1 abelian} above, we have that $g^{p^{e-1}}=a_{1}^{p^{e-1}\lambda_{1}} \dots a_{r-1}^{p^{e-1}\lambda_{r-1}}$, and so $g^{p^{e-1}} \in \langle a_{1}^{p^{b_{1}}}, \dots, a_{r-1}^{p^{b_{r-1}}} \rangle $, thus $d(H) \neq r$, and we have reached a contradiction.
\end{proof}

Combining this with the Burnside Basis Theorem \cite[Theorem 12.2.1]{hall1999theory}  gives the following result.
\begin{corollary}
If $G$ is a powerful $p$-group and $H \leq G$, then $d(H) \leq d(G).$
\end{corollary}

\section*{Acknowledgements}
The author gratefully acknowledges the many helpful suggestions of Tim Burness during the preparation of the paper. The author also wishes to express his thanks to Gareth Tracey for many helpful discussions and for suggesting Question \ref{question about burnside}. 

\bibliographystyle{amsplain} 
\bibliography{bibliography}
\end{document}